\documentclass[preprint,12pt]{elsarticle}

\usepackage{algorithmic}
\usepackage{algorithm}
\usepackage{graphicx}
\usepackage{amsmath,amsfonts,amssymb}
\usepackage{epsfig,makecell,float}
\usepackage{setspace,mathrsfs}
\usepackage{tocloft}
\usepackage{textcomp}
\usepackage{multirow,indentfirst,times,color}
\usepackage[caption=false,font=footnotesize]{subfig}
\usepackage{booktabs}
\usepackage{threeparttable}
\usepackage{amsthm}
\usepackage{extarrows}
\usepackage{bm}

\renewcommand{\arraystretch}{1.5}
\usepackage[titletoc]{appendix}
\usepackage{epstopdf}
\usepackage{natbib}

\usepackage{todonotes}
\usepackage{hyperref}

\biboptions{numbers,sort&compress}
\newtheorem{thm}{Theorem}

\newtheorem{remark}{Remark}

\newtheorem{corollary}{Corollary}
\newtheorem{example}{Example}

\newtheorem{id}{Identity}

\begin{document}

\begin{frontmatter}



\title{The $2p$ order Heisenberg-Pauli-Weyl uncertainty principles related to the offset linear canonical transform} 

\author[a,b]{Jia-Yin Peng} 
\author[a,b]{Bing-Zhao Li\corref{mycorrespondingauthor}}
\cortext[mycorrespondingauthor]{Corresponding author}\ead{li\_bingzhao@bit.edu.cn}

\affiliation[a]{organization={School of Mathematics and Statistics, Beijing Institute of Technology},
	city={Beijing},
	postcode={100081}, 
	country={China}}

\affiliation[b]{organization={Beijing Key Laboratory on MCAACI, Beijing Institute of Technology},
	city={Beijing},
	postcode={100081}, 
	country={China}}

\begin{abstract}
The uncertainty principle is one of the fundamental tools for time-frequency analysis in signal processing, revealing the intrinsic trade-off between time and frequency resolutions. With the continuous development of various advanced time-frequency analysis methods based on the Fourier transform, investigating uncertainty principles associated with these methods has become one of the most interesting topics. This paper studies the uncertainty principles related to the offset linear canonical transform, including the Plancherel-Parseval-Rayleigh identity, the $2p$ order Heisenberg-Pauli-Weyl uncertainty principle and the sharpened Heisenberg-Weyl uncertainty principle. Numerical simulations are also proposed to validate the derived results.
\end{abstract}




\end{frontmatter}

\section{Introduction}
\label{Intro}
Heisenberg inequality \cite{heisenberg,weyl,kennard} was first proposed by the German physicist Heisenberg in 1927, which indicates that in a quantum system, the position and momentum of a particle cannot be precisely measured simultaneously. In 1928, H. Weyl and W. Pauli demonstrated that the energy of a signal and that of its Fourier transform (FT) also satisfy the Heisenberg inequality. The FT \cite{FT,FT2,FT3} is one of the most significant analytical tools in signal processing, enabling the transformation of signals from the time domain to the frequency domain. Consequently, the Heisenberg inequality reveals an inherent limitation in the simultaneous concentration of energy in both the time and frequency domains \cite{HPW}, which is commonly recognized as the Heisenberg uncertainty principle in signal processing. This principle states that the time duration $\Delta_{f(t)}$ of a signal $f(t)$ and the frequency bandwidth $\Delta_{\hat{f}(\xi)}$ of its FT $\hat{f}(\xi)$ is no less than $1/2$ \cite{UN}.

This uncertainty principle conveys two valuable insights \cite{xu}. First, the resolutions in the time and frequency domains cannot be simultaneously increased without bound, as their product is constrained by a nonzero lower limit. Second, there exists a trade-off relationship between time and frequency resolution, meaning that improving frequency resolution inevitably leads to a reduction in time resolution, and vice versa. Time resolution and frequency resolution are key parameters in signal detection. High time resolution facilitates precise localization, while high frequency resolution enables accurate estimation of spectral components, such as velocity or modulation characteristics. Therefore, the uncertainty principle has become one of the important topics in signal processing \cite{UP1, UP2, UP3}. 

Building on the above research, John Michael Rassias proposed the generalized Heisenberg-Pauli-Weyl (HPW) uncertainty principle in 2004 \cite{HPW}.

In the following year, he presented the Heisenberg-Weyl (HW)  uncertainty principle \cite{HW}. To obtain a more precise bound for practical applications, Rassias proposed the sharpened Heisenberg-Weyl (SHW) uncertainty principle in the Fourier analysis in 2006 \cite{SHW}.

The above uncertainty principles are all established based on the FT \cite{FT,FT2,FT3}, which can convert time domain signals into the frequency domain to facilitate the analysis of their frequency characteristics. However, for non-stationary signals, the FT falls short in providing optimal time-frequency resolution, as it cannot effectively capture the temporal variations of the spectrum. To address this limitation, a range of advanced time-frequency analysis tools such as the fractional Fourier transform (FrFT) \cite{FrFt, FrFt2, FrFt3}, the linear canonical transform (LCT) \cite{LCT, LCT2} and the offset linear canonical transform (OLCT) have been developed. 

The OLCT is the generalization of the FT, FrFT and LCT \cite{OLCT, OLCTtwo}. It has six parameters $(a, b, c, d, \tau, \eta)$, which makes it more flexible. 

In recent years, various uncertainty principles based on the FrFT \cite{FrFtUN, FrFtUN2, FrFtDP}, LCT \cite{LCTUN, DB, UPDP} and OLCT \cite{OLCTUN1, OLCTUN2, OLCTUN3} have been extensively investigated. However, the $2p$ order HPW uncertainty principles in the OLCT domain remain unexplored, highlighting the need for further research in this area. The OLCT has demonstrated considerable theoretical and practical potential in applications such as time-frequency analysis, filter design, and target detection \cite{OLCTthree}. Therefore, studying the uncertainty principle associated with the OLCT is of both theoretical and practical importance.

In this paper, three types of the Heisenberg uncertainty principles related to the OLCT are proposed. Section \ref{sec2} is the preliminaries of this work. In section \ref{sec3}, we extend the Plancherel-Parseval-Rayleigh (PPR) identity,  the $2p$ order HPW uncertainty principle, HW uncertainty principle and the SHW uncertainty principle in OLCT domain. Section \ref{sec4} gives experimental validations of the proposed theoretical results. Section \ref{sec5} concludes this paper.

\section{Preliminaries}
\label{sec2}
In this section, we briefly introduce the definition of the OLCT and the fundamental concepts required for deriving the main results. Additionally, we review the $2p$ order Heisenberg uncertainty principle based on the FT.
\subsection{The definition of the OLCT} 
The OLCT of a signal $f(t)$ with parameter
$\renewcommand{\arraystretch}{0.8}
J = \left[\begin{array}{cc|c}
a&b&\tau\\
c&d&\eta\\
\end{array}\right]$, is defined as follows   \cite{OLCT}

\begin{equation}
\label{1}
\begin{aligned}
O_f^J(\xi)=O_J[f(t)](\xi)
=\begin{cases}
\int_{-\infty}^{+\infty} f(t) K_J(t,\xi) \mathrm{d}t,b \neq 0\\
\sqrt{d} \mathrm{e}^{\mathrm{j} [\frac{cd(\xi-\tau)^2}{2}+\xi \eta]} f(d\xi-d\tau),b=0
\end{cases},
\end{aligned}
\end{equation}
where $K_J(t,\xi)=\frac{1}{\sqrt{\mathrm{j} 2\pi b}} \mathrm{e}^{\mathrm{j} [\frac{a}{2b} t^2-\frac{1}{b}t(\xi-\tau)-\frac{1}{b}\xi (d\tau -b \eta)+\frac{d}{2b}(\xi^2+\tau^2)]}$
and parameters $a$, $b$, $c$, $d$, $\tau$, $\eta \in \mathbb{R}$, $ad-bc=1$.
When the parameter $\renewcommand{\arraystretch}{0.8}
J = \left[\begin{array}{cc|c}
a&b&0\\
c&d&0\\
\end{array}\right]$, the OLCT becomes LCT. When the parameter $\renewcommand{\arraystretch}{0.8}
J = \left[\begin{array}{cc|c}
cos(\alpha)&sin(\alpha)&0\\
-sin(\alpha)&cos(\alpha)&0\\
\end{array}\right]$, the OLCT becomes FrFt. When the parameter $\renewcommand{\arraystretch}{0.8}
J = \left[\begin{array}{cc|c}
0&1&0\\
-1&0&0\\
\end{array}\right]$, the OLCT becomes FT. 
\subsection{Some identities} 
This section presents several fundamental equations used throughout this paper. We begin by introducing a differential identity \cite{HPW}.

\begin{id}
	\label{id1}
	Suppose that $f(t)$ is a signal with $t \in \mathbb{R}$, then  for any fixed but arbitrary integer $k \in \mathbb{N}^+ = \{1,2,...\}$, the following identity holds
	\begin{equation}
	\label{6}
	f(t) \overline{f^{(k)}}(t)+f^{(k)}(t)\overline{f(t)} = \sum_{l=0}^{{[\frac{k}{2}]}} (-1)^l \frac{k}{k-l} \binom{k-l}{l} \frac{\mathrm{d}^{k-2l}}{\mathrm{d}t^{k-2l}} \lvert f^{(l)}(t) \rvert^2,
	\end{equation}
	where $f^{(l)} (t)= \frac{\mathrm{d}^l}{\mathrm{d}t^l}f(t)$ is the $l$-th derivative of $f(t)$, $l \in \mathbb{N} = {0,1,2,...}$, with $0 \leq l \leq [ \frac{k}{2} ]$, and $ [ \ ]$ denotes the floor function (i.e., the greatest integer less than or equal to its argument), $\overline{(\cdot)}$ is the conjugate of $(\cdot)$.
\end{id}

Next, we introduce a Lagrange type differential identity \cite{HPW}.
\begin{id}
	Suppose that $f_{\alpha}(t) = \mathrm{e}^{-\mathrm{j} \alpha t}f(t)$, with $\alpha = 2 \pi \xi_m$, for any fixed but arbitrary real constant $\xi_m$, the following identity holds
	\begin{equation}
	\label{9}
	\lvert f_{\alpha}^{(q)}(t) \rvert^2 = \sum_{n=0}^q B_{qn} \lvert f^{(n)}(t) \rvert^2 +2 \sum_{0 \leq i \leq z \leq q} C_{qiz} \mathrm{Re}((-1)^{q-\frac{i+z}{2}}  f^{(i)}(t) \overline{f^{(z)}(t)}).
	\end{equation}
	where 
	\begin{equation}
	\label{90}
	B_{qn}= \binom{q}{n}^2 \alpha^{2(q-n)}, C_{qiz} = s_{qi} \binom{q}{i} \binom{q}{z} \alpha^{2q-z-i}, s_{qi} \in \{\pm1\},
	\end{equation}
	and $\mathrm{Re} \overline{(\cdot)}$ is the real part of $\overline{(\cdot)}$.
\end{id}
Subsequently, we introduce an integral identity  \cite{HPW}.
\begin{id}
	Suppose that $h(t)$, $\omega(t)$ and $\omega_p(t)$ are real-valued functions and all the following integrals exist. For any fixed but arbitrary $p \in \mathbb{N}$, $i =p-2q \in \mathbb{N}^+$ and $r \in \{0, 1, 2, . . ., i-1 \}$, the following identity holds
	\begin{equation}
	\label{11}
	\int \omega_p(t) h^{(i)} (t) \mathrm{d}t = \sum_{r=0}^{i-1} (-1)^r \omega_p^{(r)}(t) h^{(i-r-1)}(t)+(-1)^i \int \omega_p^{(i)}(t) h(t) \mathrm{d}t,
	\end{equation}
	\begin{equation}
	\label{12}
	\int \omega_p(t) h^{(i)} (t) \mathrm{d}t = (-1)^i \int_{\mathbb{R}} \omega_p^{(i)}(t) h(t) \mathrm{d}t,
	\end{equation}
	where 
	\begin{equation}
	\label{91}
	\omega_p(t) = (t-t_m)^p \omega(t), \  \mathrm{for \ any \ fixed \ but \ arbitrary \ constant} \ t_m \in \mathbb{R}.
	\end{equation}
	Equation \eqref{12} holds if the condition
	\begin{equation}
	\label{13}
	\sum_{r=0}^{v-1} (-1)^r \lim_{\lvert t \rvert \rightarrow \infty} \omega_p^{(r)} (t) h^{(i-r-1)}(t) = 0,
	\end{equation}
	holds.
\end{id}
We define the FT $\hat{f}(\xi)$ of the signal $f(t)$ as follows
\begin{equation}
\label{3}
\hat{f}(\xi) = \int_{\mathbb{R}} \mathrm{e}^{-2 \mathrm{j} \pi \xi t} f(t) \mathrm{d}t,
\end{equation}  
and
\begin{equation}
\label{4}
f(t) = \int_{\mathbb{R}} \mathrm{e}^{2 \mathrm{j} \pi \xi t} \hat{f}(\xi) \mathrm{d}\xi,
\end{equation}
then we introduce the PPR identity in the FT domain \cite{HPW}.
\begin{id}
	Suppose that $f(t)$, $f_{\alpha}^{(p)}(t)$, and $(\xi-\xi_m)^p\hat{f}(\xi)$ are in $L^1(\mathbb{R}) \cap L^2(\mathbb{R})$, $\xi_m$ are any fixed but arbitrary constants. For any fixed but arbitrary $p \in \mathbb{N}$, the following identity holds
	\begin{equation}
	\label{5}
	\int_{\mathbb{R}}(\xi-\xi_m)^{2p} \lvert \hat{f}(\xi) \rvert ^2 \mathrm{d}\xi = \frac{1}{(2\pi )^{2p}} \int_{\mathbb{R}} \lvert f_{\alpha}^{(p)}(t) \rvert^2 \mathrm{d}t.
	\end{equation}
\end{id}

\subsection{The Heisenberg uncertainty principle related to the FT} 
This section introduces some Heisenberg uncertainty principles associated with the FT. First, we define
\begin{equation}
\label{51}
(\mu_{2p})_{\omega,\lvert f(t) \rvert^2} = \int_{\mathbb{R}} \omega^2(t) (t-t_m)^{2p} \left|f(t) \right|^2 \mathrm{d}t,
\end{equation}
and
\begin{equation}
\label{52}
(\mu_{2p})_{\lvert \hat{f}(\xi) \rvert ^2} = \int_{\mathbb{R}} (\xi-\xi_m)^{2p} \left| \hat{f}(\xi) \right| ^2 \mathrm{d} \xi,
\end{equation}
then we review the classical $2p$ order HPW uncertainty principle \cite{HPW}.

Suppose that $f(t) \in L^2(\mathbb{R})$ and $\omega(t)$ is a real-valued weight function. Then the $2p$ order HPW uncertainty principle in the FT domain is given by 
\begin{equation}
\label{2.2.1}
\sqrt[2p]{(\mu_{2p})_{\omega,\lvert f(t) \rvert^2}}  \sqrt[2p]{(\mu_{2p})_{\lvert \hat{f} (\xi) \rvert ^2}} \ge \frac{1}{2\pi\sqrt[p]{2}} \sqrt[p]{ \lvert E_{p,f} \rvert},
\end{equation}
where
\begin{equation}
\label{75}
E_{p,f} = \sum_{q=0}^{\left[ \frac{p}{2} \right]} D_q F_q, \ \ 0 \leq q \leq \left[ \frac{p}{2} \right], \ \left| F_q \right| < \infty,
\end{equation}

\begin{equation}
D_q = (-1)^q \frac{p}{p-q}\binom{p-q}{q}, F_q = \sum_{n=0}^q B_{qn} I_{qn} +2\sum_{0 \leq i < z \leq q} C_{qiz} I_{qiz}.
\end{equation}
Here
\begin{equation}
I_{qn} = (-1)^{q-2p} \int_{\mathbb{R}} \omega_p^{(p-2q)} (t) \left| f^{(n)} (t) \right|^2 \mathrm{d}t, 
\end{equation}
\begin{equation}
I_{qiz} = (-1)^{q-2p} \int_{\mathbb{R}} \omega_p^{(p-2q)} (t) \mathrm{Re} \left((-1)^{q - \frac{i + z}{2}}  f^{(i)}(t) \overline{f^{(z)} (t)} \right) \mathrm{d}t, 
\end{equation}
$B_{qn}$ and $C_{qiz} $ are given by \eqref{90}. 
In addition, we assume the following two conditions
\begin{equation}
\sum_{m=0}^{p-2q-1} (-1)^m \lim_{\left| t \right| \rightarrow \infty} \omega_p^{(m)} (t) \left( \left|f^{(n)} (t)\right|^2 \right) ^{p-2q-m-1} =0,
\end{equation}
and
\begin{equation}
\sum_{m=0}^{p-2q-1} (-1)^m \lim_{\left|t \right| \rightarrow \infty} \omega_p^{(m)} (t) \left( \mathrm{Re} (-1)^{q-\frac{i+z}{2}} f^{(i)} (t) \overline{f^{(z)} (t)} \right) =0.
\end{equation}

Next, let
\begin{equation}
\label{92}
(\mu_{p})_{\left| f(t) \right|^2}=\int_{\mathbb{R}} \left| t-t_m \right|^{p} \left| f(t) \right|^2 \mathrm{d}t,
\end{equation}
and
\begin{equation}
\label{100}
(\mu_{p})_{\left| \hat{f}(\xi) \right|^2}=\int_{\mathbb{R}} \left| \xi-\xi_m \right|^{p} \left| \hat{f}(\xi) \right|^2 \mathrm{d}t,
\end{equation}
then we introduce the HW uncertainty principle \cite{HW}.

Suppose that $E_{\left| f(t) \right|^2} = \int_{\mathbb{R}} \left| f(t) \right|^{2} \mathrm{d}t$ and $p \ge 2$, then the HW uncertainty principle in the FT domain is as
\begin{equation}
\label{30}
\sqrt[p]{(\mu_{p})_{\left| f(t) \right|^2}} \sqrt[p]{(\mu_{p})_{\left| \hat{f}(\xi) \right|^2}} \ge \frac{1}{4 \pi} \sqrt[p]{E_{\left| f(t) \right|^2}^2}.
\end{equation}

Then we present the SHW uncertainty principle \cite{SHW}.

Suppose that $f(t) \in L^2(R)$, and for any fixed but arbitrary $p \in \mathbb{N}$, the SHW uncertainty principle  associated with the FT  holds as follows
\begin{equation}
\label{2}
\begin{aligned}
\sqrt[2p]{(\mu_{2p})_{\omega,\lvert f(t) \rvert^2}} \sqrt[2p]{(\mu_{2p})_{\lvert \hat{f} (\xi)  \rvert ^2}} \geq \frac{1}{2 \pi \sqrt[p]{2}} \sqrt[p]{\lvert E_{p,f}^{\star} \rvert},
\end{aligned}
\end{equation}
where 
\begin{equation}
\label{54}
\left| E_{p,f}^{\star} \right| = \sqrt{E_{p,f}^2+4{A^{\star}}^2},
\end{equation}
and $(\mu_{2p})_{\omega,\lvert f(t) \rvert^2}$ is given by \eqref{51}, $(\mu_{2p})_{\lvert \hat{f} (\xi) \rvert}$ is given by \eqref{52}, $\left| E_{p,f} \right| $ is given by \eqref{75}.
The equality in \eqref{2} holds when 
\begin{equation}
A^{\star} =|| u^{\star}(t) || {x_0}^{\star} - || {v^{\star}(t)} || {y_0}^{\star},
\end{equation}
where $u^{\star}(t) = \omega(t) (t-t_m)^p f_{\alpha}(t)$, $v^{\star}(t) = f_{\alpha}^{(p)} (t)$ and ${x_0}^{\star} = \int_{\mathbb{R}} \left|v^{\star}(t) \right| \left| h(t) \right| \mathrm{d}t$, $ {y_0}^{\star} = \int_{\mathbb{R}} \left|u^{\star}(t) \right| \left| h(t) \right| \mathrm{d}t$, $\left\|h(t)\right\|^2 = \int_{\mathbb{R}} \left|h(t) \right|^2 \mathrm{d}t =1$, $t_m$ are any fixed but arbitrary constants.

\section{Main Results}
\label{sec3}
This section derives the main theorems of this paper, including the PPR identity, the $2p$ order HPW uncertainty principle, the HW uncertainty principle, and the SHW uncertainty principle in the OLCT domain.
\subsection{The PPR identity of the OLCT}
To derive the $2p$ order Heisenberg uncertainty principles in the OLCT domain, we first derived the PPR identity, which can be represented as the following Theorem \ref{thm3.1}.
\begin{thm}
	\label{thm3.1}
	Suppose that $f(t)$, $O_f^J(\xi)$, $(\xi-\xi_m)^p O_f^J(\xi)$ belong to $L^1(\mathbb{R}) \cap L^2(\mathbb{R})$, $\xi_m$ are any fixed but arbitrary constants. For any fixed but arbitrary $p \in \mathbb{N} = \{0,1,2,...\}$, we have
	\begin{equation}
	\label{14}
	\int_{\mathbb{R}} (\xi-\xi_m)^{2p} \lvert O_f^J(\xi) \rvert^2 \mathrm{d} \xi = b^{2p} \int_{\mathbb{R}} \lvert g_\beta^{(p)}(t) \rvert^2 \mathrm{d}t,
	\end{equation}
	where $g_{\beta}^{(p)}(t) = \frac{d^p}{dt^p} \mathrm{e}^{-\mathrm{j} \beta t} \mathrm{e}^{\mathrm{j} \frac{a}{2b} t^2}  f(t)$ with $\beta=2 \pi \sigma_m$, $\sigma_m = \frac{\xi_m-\tau}{2 \pi b}$. 
\end{thm}

\begin{proof}
	Regarding the left-hand side of \eqref{14}, it can be rewritten as
	\begin{align*}
	&\int_{\mathbb{R}} (\xi-\xi_m)^{2p} \lvert O_f^J(\xi) \rvert^2 d \xi \nonumber\\
	=& \int_{\mathbb{R}} (\xi-\xi_m)^{2p} \left| \frac{1}{\sqrt{\mathrm{j} 2\pi b}} \int_{\mathbb{R}} f(t) \mathrm{e}^{\mathrm{j}[\frac{a}{2b} t^2-\frac{1}{b}t(\xi-\tau)-\frac{1}{b}\xi(d\tau -b \eta)+\frac{d}{2b}(\xi^2+\tau^2)]} \mathrm{d}t \right|^2 \mathrm{d} \xi \nonumber\\
	=&\frac{1}{2 \pi b} \int_{\mathbb{R}}(\xi-\xi_m)^{2p} \left| \int_{\mathbb{R}} f(t) \mathrm{e}^{\mathrm{j} [\frac{a}{2b} t^2-\frac{(\xi-\tau)t}{b}]}dt \right|^2 \mathrm{d}\xi.
	\end{align*}
	Since $g(t) = f(t)\mathrm{e}^{\mathrm{j} \frac{a}{2b} t^2}$, then
	\begin{align}
	\label{16}
	&\int_{\mathbb{R}} (\xi-\xi_m)^{2p} \left| O_f^J(\xi) \right|^2 \mathrm{d} \xi \nonumber\\
	=&\frac{1}{2 \pi b} \int_{\mathbb{R}}(\xi-\xi_m)^{2p} \left| \int_{\mathbb{R}} g(t) \mathrm{e}^{-\mathrm{j} \frac{(\xi-\tau)}{b}t}dt \right|^2 \mathrm{d} \xi \nonumber\\
	=&\frac{1}{2 \pi b} \int_{\mathbb{R}}(\xi-\xi_m)^{2p} \left| \hat{g}(\frac{\xi-\tau}{2 \pi b})\right|^2 \mathrm{d} \xi \nonumber\\
	=&\frac{1}{2 \pi b} \int_{\mathbb{R}} [\xi-\tau-(\xi_m-\tau)]^{2p} \left| \hat{g}(\frac{\xi-\tau}{2 \pi b}) \right|^2 \mathrm{d}( \xi-\tau) \nonumber \\
	=&\frac{1}{2 \pi b} \int_{\mathbb{R}} (2 \pi b)^{2p} (\sigma-\sigma_m)^{2p} \left| \hat{g}(\sigma) \right|^2 \mathrm{d}(2 \pi b \sigma) \nonumber\\
	=&(2 \pi b)^{2p} \int_{\mathbb{R}} (\sigma-\sigma_m)^{2p} \left| \hat{g}(\sigma) \right|^2 \mathrm{d}\sigma.
	\end{align}
	Using \eqref{5}, we obtain
	\begin{equation}
	\label{76}
	\int_{\mathbb{R}} (\sigma-\sigma_m)^{2p} \left| \hat{g}(\sigma) \right|^2 \mathrm{d}\sigma = \frac{1}{(2\pi )^{2p}} \int_{\mathbb{R}} \lvert g_{\beta}^{(p)}(t) \rvert^2 \mathrm{d}t.
	\end{equation}
	Substituting \eqref{76} into \eqref{16} completes the proof.
\end{proof}

\subsection{The $2p$ order HPW uncertainty principle of the OLCT}
We define
\begin{equation}
\label{53}
(\mu_{2p})_{\left| O_f^J(\xi) \right|^2} = \int_{\mathbb{R}} (\xi-\xi_m)^{2p} \left|O_f^J(\xi) \right|^2 \mathrm{d}\xi,
\end{equation}
to establish the $2p$ order HPW uncertainty principle in the OLCT domain as following Theorem \ref{thm3.2}. 
\begin{thm}
	\label{thm3.2}
	Suppose that $f(t)$, $O_f^J(\xi)$, $(\xi-\xi_m)^p O_f^J(\xi)$ belong to $L^1(\mathbb{R}) \cap L^2(\mathbb{R})$. Let $\omega(t)$ be a real weight function. For any fixed but arbitrary $p \in \mathbb{N}$, the $2p$ order HPW uncertainty principle in the OLCT domain is as follows
	\begin{equation}
	\label{17}
	\sqrt[2p]{(\mu_{2p})_{\omega,\left|f(t) \right|^2}} \sqrt[2p]{(\mu_{2p})_{\left| O_f^J(\xi) \right|^2}} \ge \frac{|b|}{\sqrt[p]{2}} \sqrt[p]{\left|E_{p,f}\right|},
	\end{equation}
	where $(\mu_{2p})_{\omega,\left|f(t) \right|^2}$ is given by \eqref{51} and $E_{p,f}$ is given by \eqref{75}.
\end{thm}
\begin{remark}
	We note that $(\mu_{2p})_{\left| O_f^J(\xi) \right|^2} = b^{2p} \int_{\mathbb{R}} \left| \frac{d^{2p}}{dt^{2p}} \left(\mathrm{e}^{-\mathrm{j} \frac{\xi_m-\tau}{b} t} \mathrm{e}^{\mathrm{j} \frac{a}{2b} t^2}  f(t) \right) \right|^2 \mathrm{d} t$, therefore $\tau$ is a crucial parameter in HPW uncertainty principle. In contrast to the HPW uncertainty principle in the LCT domain, Theorem \ref{thm3.2} involves the offset parameter $\tau$ in the OLCT, providing additional flexibility.
\end{remark}
\begin{corollary}
	\label{co1}
	For $p=1$, the second-order HPW uncertainty principle in the OLCT domain can be obtained as follows
	\begin{equation}
	\label{70}
	\sqrt{(\mu_{2})_{\omega,\left|f(t) \right|^2}} \sqrt{(\mu_{2})_{\left| O_f^J(\xi) \right|^2}} \ge \frac{|b|}{2} \sqrt{E_{1,f}},
	\end{equation}
	where
	\begin{equation}
	E_{1,f} = \int_{\mathbb{R}} \left[ \left(t-t_m \right) \omega(t) \right]^{'} \left|f(t)\right|^2 \mathrm{d}t.
	\end{equation}
	Equality holds in \eqref{70} if and only if
	\begin{equation}
	\label{72}
	f(t) = c_0 \mathrm{e}^{-c_p(t-t_m)^2}\mathrm{e}^{\mathrm{j}(\frac{\xi_m}{b}t-\frac{a}{2b}t^2)},
	\end{equation}
	where $c_0$, $t_m$, $\xi_m \in \mathbb{R}$ and $c_p \in \mathbb{R}^+$.
\end{corollary}

\begin{corollary}
	For $p=1$, assuming $\omega (t)= 1$ and $ ||f(t)||_2^2 = \int_{\mathbb{R}} \left| f(t) \right|^2 \mathrm{d}t = E_{\left|f \right|^2}$, the second-order HPW uncertainty principle in the OLCT domain can be derived as follows
	\begin{equation}
	\label{86}
	(\mu_{2})_{1,\left|f(t) \right|^2}(\mu_{2})_{\left| O_f^J(\xi) \right|^2} \ge \frac{b^{2}}{4} ||f||_2^4.
	\end{equation}
	When $ ||f(t)||_2 = 1$ and $\xi_m=\int_{\mathbb{R}}\xi \left|O_f^J(\xi) \right|^2 \mathrm{d}\xi$, equation \eqref{86} reduces to the Heisenberg uncertainty principle derived in \cite{OLCTUN0}.
\end{corollary}

\begin{corollary}
	When the parameters of the OLCT are set to $\renewcommand{\arraystretch}{0.8}
	J = \left[\begin{array}{cc|c}
	0&1&0\\
	-1&0&0\\
	\end{array}\right]$, Theorem \ref{thm3.2} reduces to the $2p$ order uncertainty principle in the FT domain given in \eqref{2.2.1}.
\end{corollary}

\subsection{The HW uncertainty principle of the OLCT}
Let
\begin{equation}
\label{57}
(\mu_{p})_{\left| O_f^J (\xi) \right|^2} = \int_{\mathbb{R}} \left| \xi - \xi_m \right|^{p} \left| O_f^J(\xi) \right|^2 \mathrm{d} \xi,
\end{equation}
then we can obtain the following Theorem \ref{thm3.3}.

\begin{thm}
	\label{thm3.3}
	Suppose that $f(t)$, $O_f^J(\xi)$, $(\xi-\xi_m)^p O_f^J(\xi)$ belong to $L^1(\mathbb{R}) \cap L^2(\mathbb{R})$. For $p \ge 2$ and any fixed but arbitrary real constants $t_m$, $\xi_m$, the HW uncertainty principle of the OLCT holds as follows
	\begin{equation}
	\label{55}
	\sqrt[p]{(\mu_{p})_{\left| f (t) \right|^2}} \sqrt[p]{(\mu_{p})_{\left| O_f^J (\xi) \right|^2}} \ge \frac{|b|}{2} \sqrt[p]{E_{\left| f(t) \right|^2}^2},
	\end{equation}
	where $(\mu_{p})_{\left| f(t) \right|^2}$ is given by \eqref{92} and $E_{\left| f(t) \right|^2} = \int_{\mathbb{R}} \left| f(t) \right|^{2} \mathrm{d}t$.
\end{thm}

\begin{proof}
	Firstly,
	\begin{equation}
	\label{83}
	\begin{aligned}
	(\mu_{p})_{\left| f(t) \right|^2}^{\frac{2}{p} }E_{\left| f(t) \right|^2}^{1-\frac{2}{p}} &= \left( \int_{\mathbb{R}} \left| t-t_m \right|^{p} \left| f(t) \right|^2 \mathrm{d}t \right)^{\frac{2}{p}} \left( \int_{\mathbb{R}} \left|f(t) \right|^2 \mathrm{d}t \right)^{1-\frac{2}{p}}\\
	&= \left[\int_{\mathbb{R}} \left( \left| t-t_m \right|^{2} \left| f(t) \right|^{\frac{4}{p}} \right) ^{\frac{p}{2}} \mathrm{d}t \right]^{\frac{2}{p}} \left[\int_{\mathbb{R}} \left( \left|f(t) \right|^{2\left(1-\frac{2}{p} \right)} \right)^{1/\left(1-\frac{2}{p} \right)} \mathrm{d}t \right]^{1-\frac{2}{p}}.
	\end{aligned}
	\end{equation}
	Then applying the Hölder inequality to \eqref{83}, we obtain
	\begin{equation*}
	\begin{aligned}
	(\mu_{p})_{\left| f(t) \right|^2}^{\frac{2}{p} }E_{\left| f(t) \right|^2}^{1-\frac{2}{p}} & \ge \int_{\mathbb{R}}  \left( (t-t_m) ^2 \left|f(t) \right|^{\frac{4}{p}} \right) \left(\left|f(t) \right|^{2\left(1-\frac{2}{p} \right) } \right)  \mathrm{d}t\\
	&= \int_{\mathbb{R}} (t-t_m) ^2 \left|f(t) \right|^2 \mathrm{d}t\\
	&= (\mu_{2})_{\left| f(t) \right|^2} \triangleq \phi_{\left|f(t) \right|^2}^2,
	\end{aligned}
	\end{equation*}
	so 
	\begin{equation}
	\label{61}
	(\mu_{p})_{\left| f(t) \right|^2}^{\frac{1}{p}} \ge  \phi_{\left|f(t) \right|^2} / E_{\left| f(t) \right|^2}^{\frac{1}{2}-\frac{1}{p}}.
	\end{equation}
	Equality holds in \eqref{61} if and only if
	\begin{equation*}
	\left| t-t_m \right|^{p} E_{\left| f(t) \right|^2} = (\mu_{p})_{\left| f(t) \right|^2}.
	\end{equation*}
	Similarly, 
	\begin{equation*}
	\begin{aligned}
	(\mu_{p})_{\left| O_f^J(\xi) \right|^2}^{\frac{2}{p} } E_{\left| O_f^J(\xi) \right|^2}^{1-\frac{2}{p}} &= \left( \int_{\mathbb{R}} \left| \xi-\xi_m \right|^{p} \left| O_f^J(\xi) \right|^2 \mathrm{d} \xi \right)^{\frac{2}{p}} \left( \int_{\mathbb{R}} \left| O_f^J (\xi) \right|^2 \mathrm{d}\xi \right)^{1-\frac{2}{p}} \\
	&\ge \int_{\mathbb{R}} (\xi-\xi_m) ^2 \left| O_f^J(\xi) \right|^2 \mathrm{d}\xi\\
	&= (\mu_{2})_{\left| O_f^J(\xi) \right|^2} \triangleq \phi_{\left| O_f^J(\xi) \right|^2}^2,
	\end{aligned}	
	\end{equation*}
	so 
	\begin{equation}
	\label{62}
	(\mu_{p})_{\left| O_f^J(\xi) \right|^2}^{\frac{1}{p}} \ge  \phi_{\left| O_f^J(\xi) \right|^2} / E_{\left| O_f^J(\xi) \right|^2}^{\frac{1}{2}-\frac{1}{p}}.
	\end{equation}
	Equality holds in \eqref{62} if and only if
	\begin{equation*}
	\left| \xi-\xi_m \right|^{p} E_{\left| O_f^J(\xi) \right|^2} = (\mu_{p})_{\left| O_f^J(\xi) \right|^2}.
	\end{equation*}
	By the Parseval property of the OLCT \cite{OLCTP}, we obtain
	\begin{equation}
	E_{\left| f(t) \right|^2} = E_{\left| O_f^J(\xi) \right|^2},
	\end{equation}
	then we find
	\begin{equation}
	\label{102}
	(\mu_{p})_{\left| f(t) \right|^2}^{\frac{1}{p}} (\mu_{p})_{\left| O_f^J(\xi) \right|^2}^{\frac{1}{p}} \ge \phi_{\left|f(t) \right|^2} \phi_{\left| O_f^J(\xi) \right|^2} / E_{\left| f(t) \right|^2}^{1-\frac{2}{p}}.
	\end{equation}
	From \eqref{86}, we have
	\begin{equation}
	\label{101}
	\phi_{\left|f(t) \right|^2}^2 \phi_{\left| O_f^J(\xi) \right|^2}^2 \ge \frac{b^2}{4} E_{\left| f(t) \right|^2} ^2.
	\end{equation}
	Substituting equation \eqref{101} into \eqref{102} completes the proof of Theorem \ref{thm3.3}.
\end{proof}

\subsection{The SHW uncertainty principle of the OLCT}
This section builds upon the $2p$ order HPW uncertainty principle in the OLCT domain and presents the SHW uncertainty principle, which provides a more precise lower bound.

\begin{thm}
	\label{thm3.4}
	Suppose that $f(t)$, $O_f^J(\xi)$, $(\xi-\xi_m)^p O_f^J(\xi)$ belong to $L^1(\mathbb{R}) \cap L^2(\mathbb{R})$. For any fixed but arbitrary $p \in \mathbb{N}$, the SHW uncertainty principle of the OLCT holds as follows
	\begin{equation}
	\label{31}
	\sqrt[2p]{(\mu_{2p})_{\omega, \left| f(t) \right|^2}} \sqrt[2p]{(\mu_{2p})_{ \left| O_f^J(\xi) \right|^2}} \ge \frac{|b|}{\sqrt[p]{2}} \sqrt[p]{\left| E_{p,f}^* \right|}.
	\end{equation}
	The equality in \eqref{31} holds when 
	\begin{equation}
	A = ||u(t)||x_0-||v(t)||y_0,
	\end{equation}
	where $u(t) = \omega(t) (t-t_m)^p g_{\beta}(t)$, $v(t) = g_{\beta}^{(p)} (t)$ and $x_0 = \int_{\mathbb{R}} \left|v(t) \right| \left| h(t) \right| \mathrm{d}t$, $y_0 = \int_{\mathbb{R}} \left|u(t) \right| \left| h(t) \right| \mathrm{d}t$, $\left\|h(t)\right\|^2 = \int_{\mathbb{R}} \left|h(t) \right|^2 \mathrm{d}t =1$.
\end{thm}

\begin{proof}
	In fact, from \eqref{14} in Theorem \ref{thm3.1}, we obtain
	\begin{align}
	\label{32}
	M_p^* =& M_p -b^{2p} A^2 \nonumber\\
	=&(\mu_{2p})_{\omega,\left|f(t) \right|^2}(\mu_{2p})_{\left| O_f^J(\xi) \right|^2} - b^{2p} A^2 \nonumber\\
	=&\left( \int_{\mathbb{R}} \omega^2(t) (t-t_m)^{2p} \left|f(t) \right|^2 \mathrm{d}t \right) \cdot \left( \int_{\mathbb{R}} (\xi-\xi_m)^{2p} \left|O_f^J(\xi) \right|^2 \mathrm{d}\xi \right) -b^{2p} A^2 \nonumber\\
	=&b^{2p} \left[ \left( \int_{\mathbb{R}} \omega^2(t) (t-t_m)^{2p} \left| g_{\beta}(t) \right|^2 \mathrm{d}t \right) \left( \int_{\mathbb{R}} \left| g_{\beta}^{(p)} (t) \right| ^2 \mathrm{d}t \right) -A^2 \right]  \nonumber\\
	=&b^{2p} \left( \left\|u\right\|^2 \left\|v\right\|^2 -A^2 \right).
	\end{align} 
	From the positive definiteness of the following Gram determinant, we obtain
	\begin{align*}
	0 \leq &\begin{pmatrix}
	\left\|u\right\|^2&(\left|u\right|,\left|v\right|)&y_0\\
	(\left|u\right|,\left|v\right|)&\left\|v\right\|^2&x_0\\
	y_0&x_0&1\\
	\end{pmatrix} \nonumber \\
	=&\left\|u\right\|^2 \left\|v\right\|^2-(\left|u\right|,\left|v\right|)^2-\left[\left\|u\right\|^2 x_0^2 -2(\left|u\right|,\left|v\right|)x_0y_0+\left\|v\right\|^2y_0^2 \right] \nonumber \\
	\leq &\left\|u\right\|^2 \left\|v\right\|^2-(\left|u\right|,\left|v\right|)^2-A^2,
	\end{align*}
	then we find
	\begin{align}
	\label{34}
	M_p^* \ge& b^{2p}(\left|u \right|,\left|v \right|)^2 \nonumber \\
	=&b^{2p} \left( \int_{\mathbb{R}} \left|u \right| \left|v \right| \mathrm{d}t \right)^2  \nonumber\\
	=&b^{2p}  \left( \int_{\mathbb{R}} \left|\omega_p(t) g_{\beta}(t) g_{\beta}^{(p)} (t) \right| \mathrm{d}t \right)^2.
	\end{align}
	The right-hand sides of equations \eqref{23} and \eqref{34} are exactly the same, so we obtain
	\begin{equation}
	\label{35}
	M_p^* \ge \frac{b^{2p}}{4} E_{p,f}^2.
	\end{equation}
	Substituting equation \eqref{35} into \eqref{32}, we have
	\begin{align}
	\label{36}
	M_p \ge& \frac{b^{2p}}{4} E_{p,f}^2+b^{2p} A^2 \nonumber\\
	=&\frac{b^{2p}}{4} \left( E_{p,f}^2+4A^2 \right).
	\end{align}
	So the $2p$ order SHW uncertainty principle of the OLCT is as follows
	\begin{equation*}
	\sqrt[2p]{M_p} \ge \frac{|b|}{\sqrt[p]{2}} \sqrt[p]{\left| E_{p,f}^{*} \right|},
	\end{equation*}
	where $\left| E_{p,f}^{*} \right| = \sqrt{E_{p,f}^2+4A^2}$.
	Hence, Theorem \ref{thm3.4} is proved.
\end{proof}

Theorem \ref{thm3.2} and Theorem \ref{thm3.4} provide two lower bounds for the HPW uncertainty principle of the OLCT. From a mathematical standpoint, a larger value on the right side of the inequality indicates a more precise and reliable estimate of the left side. Therefore, it is crucial to study the sharpened version of the uncertainty principle. From a practical signal processing perspective, a smaller lower bound implies that higher time-frequency concentration is theoretically permissible in time-frequency analysis, thereby characterizing the ideal theoretical limit. However, it remains necessary to derive sharpened version of the uncertainty principle in the OLCT domain with larger lower bounds. The reason is that a smaller theoretical bound merely reflects the potential limit under ideal conditions, whereas actual signals often fail to attain this level of concentration. In contrast, sharpened lower bounds, although numerically larger, provide more precise estimates that are closer to the true achievable limits. 

\section{Simulations}
\label{sec4}
This section validates the results of this paper through numerical experiments. Taking Theorem \ref{thm3.4} as an example, the superiority of the numerical results is demonstrated and the factors influencing signal concentration in the OLCT domain are analyzed.
\begin{example}
	\label{ex1}
	Consider a signal $f(t)=\mathrm{e}^{-\frac{r}{2}t^2} \mathrm{e}^{-\mathrm{j}\frac{a}{2b}t^2}$ and a real weight function $\omega(t) = \mathrm{e}^{-rt}$ with $r > 0$.
	
	We let $Q_1 = \sqrt[2p]{(\mu_{2p})_{\omega, \left| f(t) \right|^2}} \sqrt[2p]{(\mu_{2p})_{ \left| O_f^J(\xi) \right|^2}}$ and $Q_2 = \frac{b}{\sqrt[p]{2}} \sqrt[p]{\left| E_{p,f}^* \right|}$. For different values of $A$, we compute $Q_1$ and $Q_2$ respectively and then plot the corresponding graphs in Figure \ref{fig1}.
	\vspace{-1.5em}
	\begin{figure}[H]
		\centering
		\subfloat[\shortstack{%
			$A = \|u\| x_0 - \|v\| y_0$,\\
			$Q_1 = Q_2$
		}]{
			\includegraphics[width=0.34\linewidth]{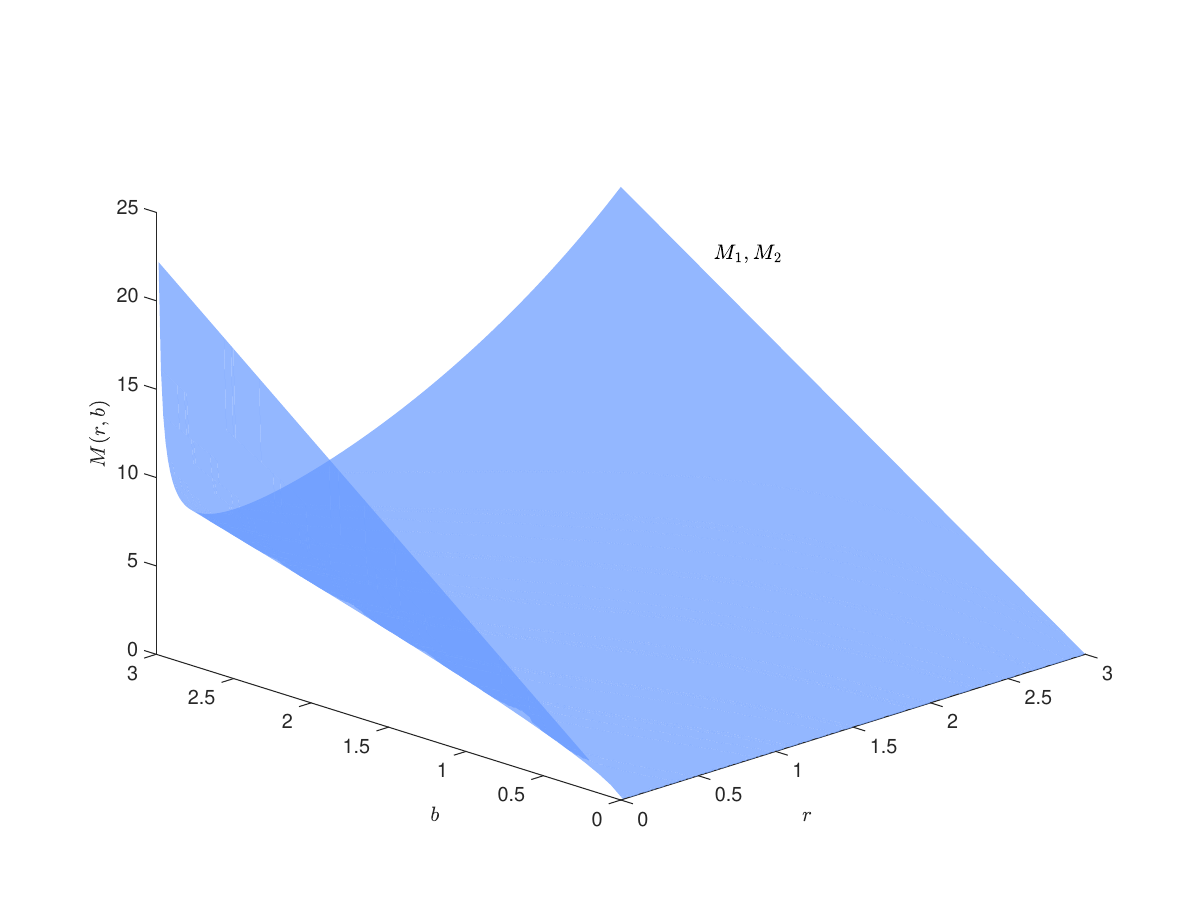}
		}
		\hspace{-1.5em}
		\subfloat[\shortstack{%
			$A=0$,\\
			$Q_1 > Q_2$
		}]{
			\includegraphics[width=0.33\linewidth]{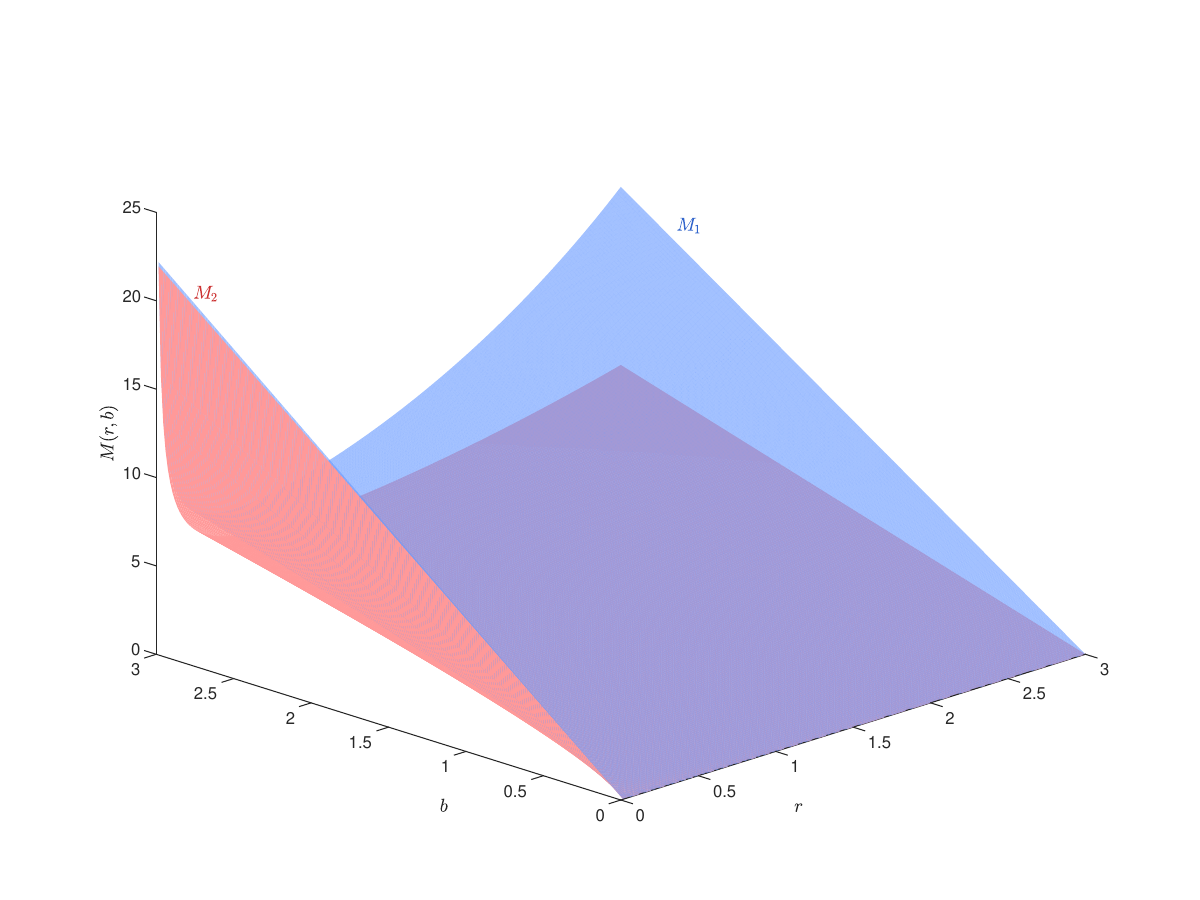}
		}
		\hspace{-1.5em}
		\subfloat[\shortstack{%
			$A=1$,\\
			$Q_1 > Q_2$
		}]{
			\includegraphics[width=0.33\linewidth]{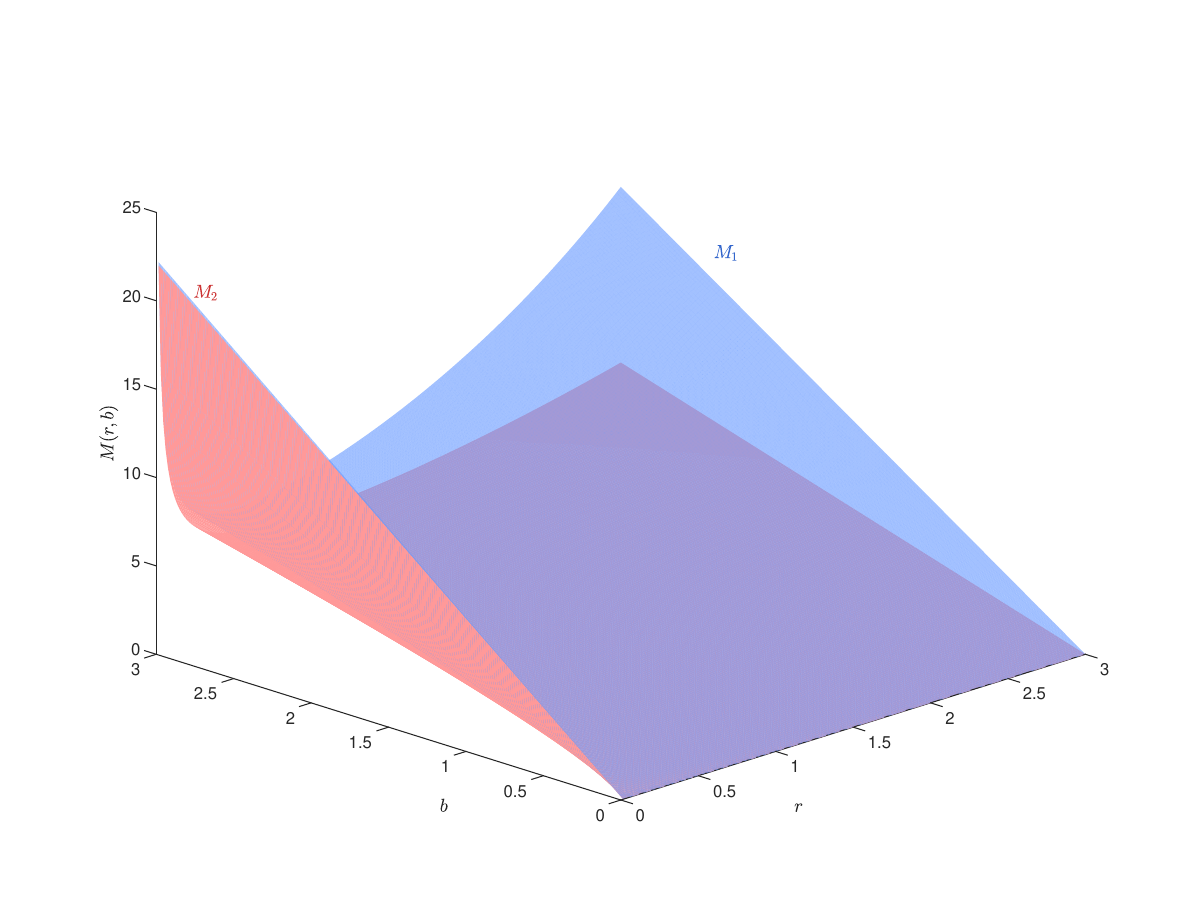}
		}
		\caption{Numerical simulations for $p=1$, $t_m=0$ and $\tau=\xi_m=0$.}
		\label{fig1}
	\end{figure}
\end{example}
Figure \ref{fig1} illustrates that when $A = ||u||x_0-||v||y_0$, we obtain $Q_1 =Q_2$. In this case, the product of the time resolution and the frequency resolution in the OLCT domain can reach its minimum value. In other words, the sharpened  HPW uncertainty principle can achieve the theoretical lower bound.

For different values of $r$, we compute $Q_1$ and $Q_2$ with $b=1$ and $A=1$, then plot the corresponding graphs in Figure \ref{fig1.1}.
\begin{figure}[H]
	\centering
	\setlength{\abovecaptionskip}{0.5pt}
	\includegraphics[width=0.76\linewidth]{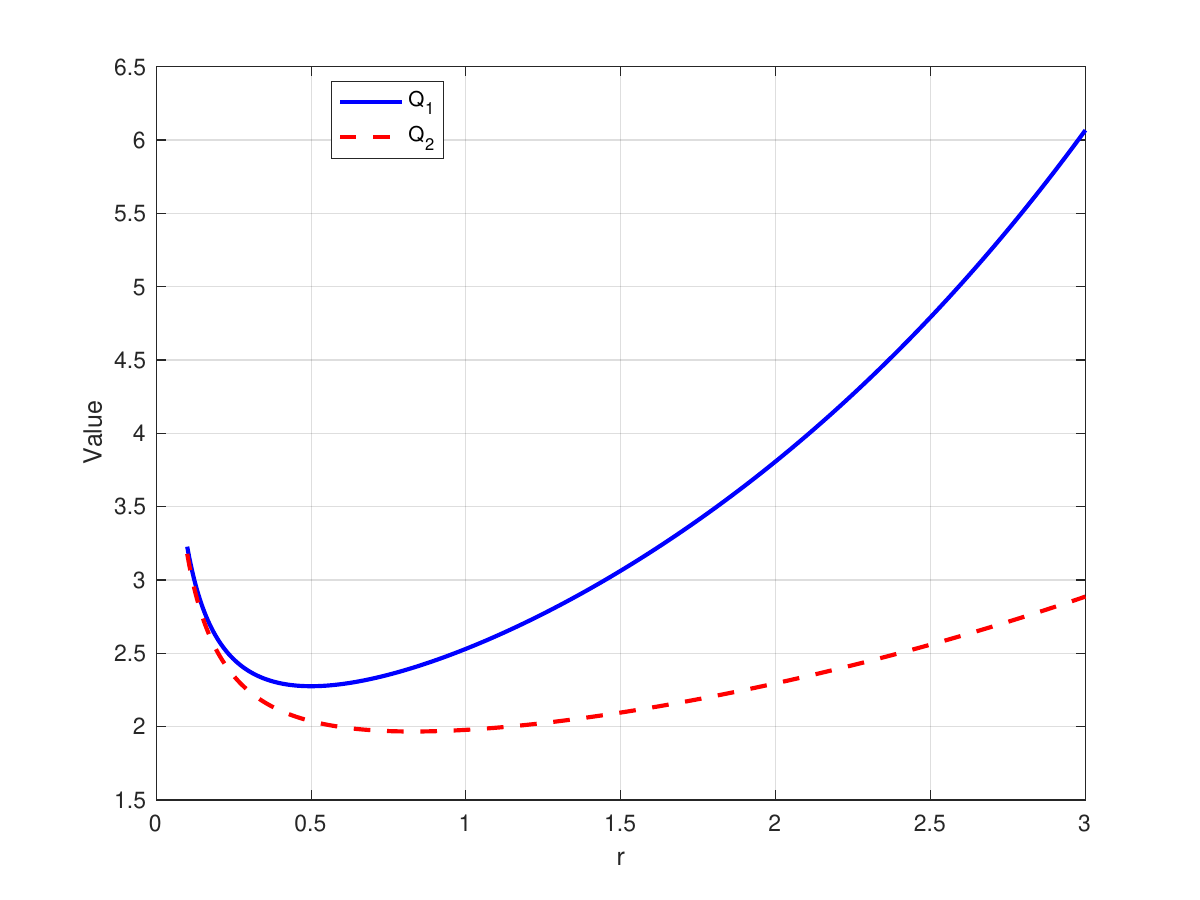}  
	\caption{Numerical simulations for $p=1$, $t_m=0$ and $\tau=\xi_m=0$.}
	\label{fig1.1}
\end{figure}
From figure \ref{fig1.1} it can be observed that $Q_1>Q_2$ holds for all $r>0$, thereby confirming the validity of Theorem \ref{thm3.4}.

Furthermore, we compare the lower bounds provided by different $2p$ order uncertainty principles in Table \ref{tab1}. The signal in Example \ref{ex1} is identical to that in \cite{LCTUN}. Hence, we proceed to compare our results with those presented in that reference.
\begin{table}[H]
	\centering
	\caption{Comparison of the lower bounds of different uncertainty principles.}
	\vspace{0.15cm}
	\label{tab1}
	\begin{tabular}{p{6cm}|p{6cm}}
		\hline \hline
		\textbf{Uncertainty principles} & \textbf{The lower bound in Example \ref{ex1}} \\
		\hline
		
		The $2p$ order uncertainty principle derived in this paper
		&
		$\frac{b^2}{2} \pi \mathrm{e}^r \left( \frac{1}{2r}+1 \right)$ \\
		\hline
		
		The $2p$ order uncertainty principle derived in \cite{LCTUN}
		&
		$\frac{b^2}{4} \frac{\pi}{r} \mathrm{e}^{\frac{r}{2}} \left( 1+\frac{r}{2} \right)^2$ \\
		\hline
	\end{tabular}
\end{table}

Compared with the result in \cite{LCTUN}, it can be concluded that for $r>0$, the inequality 
\begin{equation}
\label{84}
\frac{b^2}{2} \pi \mathrm{e}^r \left( \frac{1}{2r}+1 \right) > 	\frac{b^2}{4} \frac{\pi}{r} \mathrm{e}^{\frac{r}{2}} \left( 1+\frac{r}{2} \right)^2, 
\end{equation}
holds. This indicates that the SHW uncertainty principle provides a tighter lower bound. Then we prove the \eqref{84} holds universally. The difference between the two sides of the inequality \eqref{84} can be explicitly expressed as
\begin{equation*}
\frac{b^2}{2} \pi \mathrm{e}^r \left( \frac{1}{2r}+1 \right) - \frac{b^2}{4} \frac{\pi}{r} \mathrm{e}^{\frac{r}{2}} \left( 1+\frac{r}{2} \right)^2 = \frac{b^2}{2} \pi \mathrm{e}^{\frac{r}{2}} \left[ \mathrm{e}^{\frac{r}{2}}\left(\frac{1}{2r}+1 \right)-\frac{1}{2r}\left(1+\frac{r}{2} \right)^2 \right].
\end{equation*}
Let $G(r) = \mathrm{e}^{\frac{r}{2}}\left(\frac{1}{2r}+1 \right)-\frac{1}{2r}\left(1+\frac{r}{2} \right)^2$ and its graph is shown in Figure \ref{fig2}.
\begin{figure}[H]
	\centering
	\includegraphics[width=0.76\linewidth]{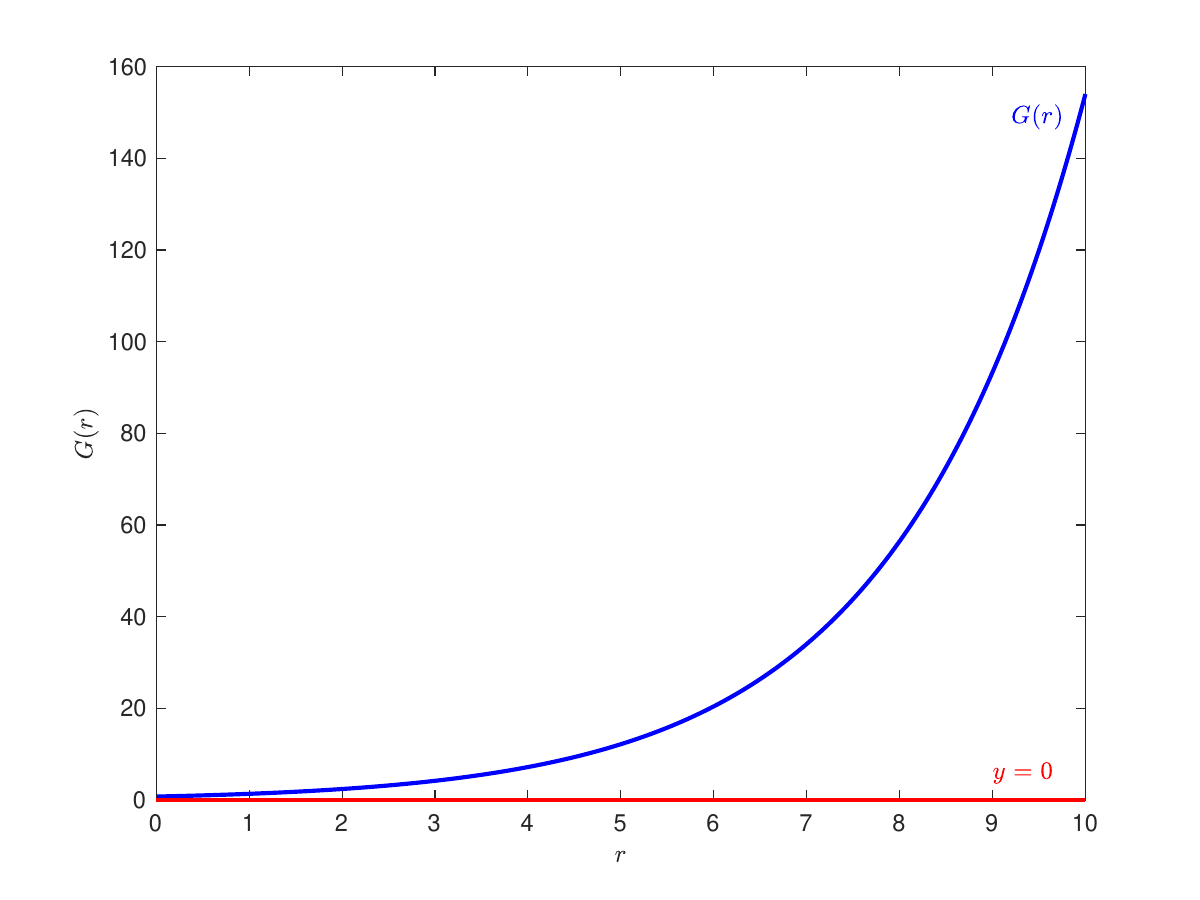}  
	\caption{The plot of $G(r)$.}
	\label{fig2}
\end{figure}
Figure \ref{fig2} clearly shows that $G(r)>0$ for all $r>0$, confirming the universal validity of inequality \eqref{84}.

\begin{example}
	Consider the signal in Example \ref{ex1} with $r=2$, then $f(t)=\mathrm{e}^{-t^2} \mathrm{e}^{-\mathrm{j}6t^2}$ and $\omega(t) = \mathrm{e}^{-2t}$. We set the parameters as $a=0.6$, $b=0.05$, $c=0.5$, $d=0.4$, $\tau=0$, $\eta=1$, $t_m=0$, $\xi_m=0$.
\end{example}
Compared with the FT, the OLCT allows for better energy concentration of signals through appropriate parameter selection. The following figure illustrates the energy distribution of the original signal, the weighted signal, the signal in the FT domain, and the signal in the OLCT domain.
\begin{figure}[H]
	\centering
	\subfloat[Energy density of $f(t)$]{ 
		\includegraphics[width=2.7 in]{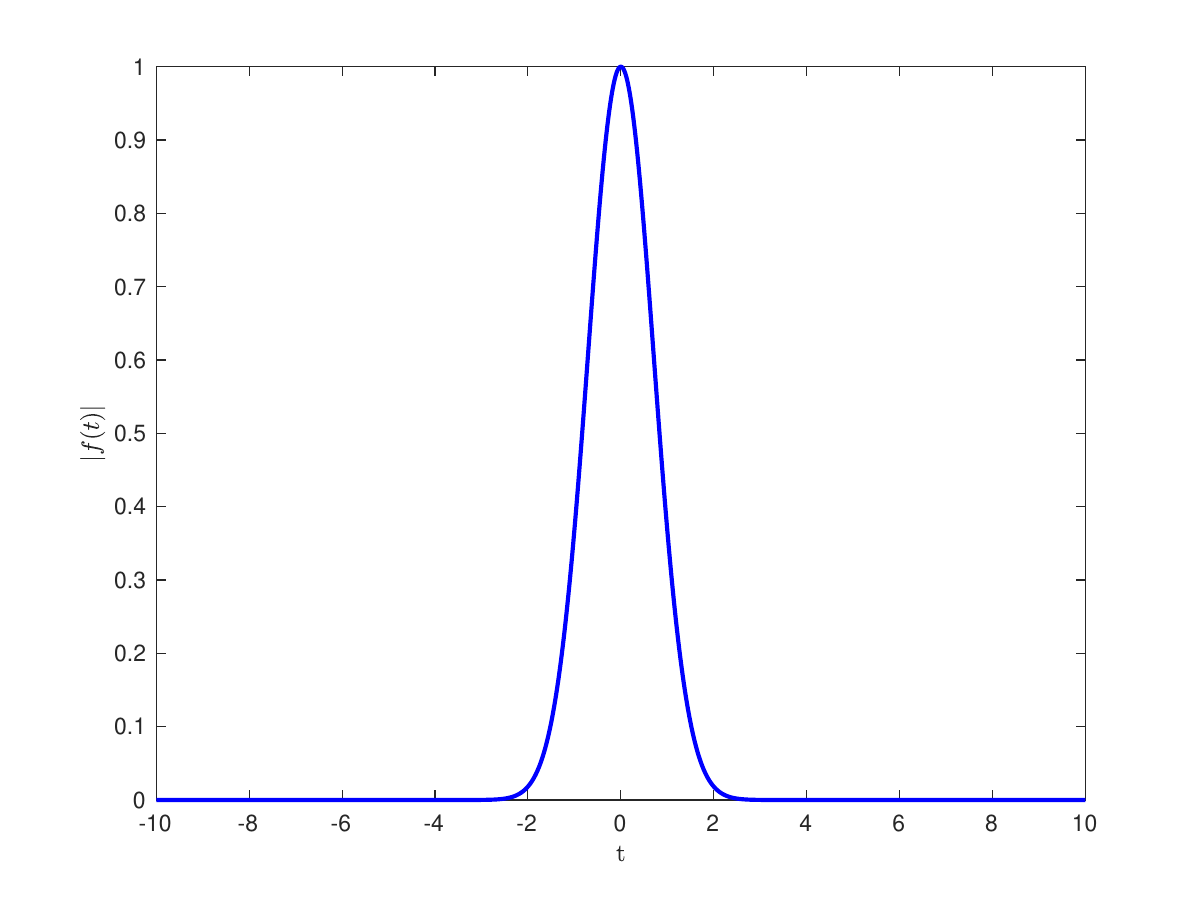}}
	\subfloat[Energy density of weighted $f(t)$]{
		\includegraphics[width=2.7 in]{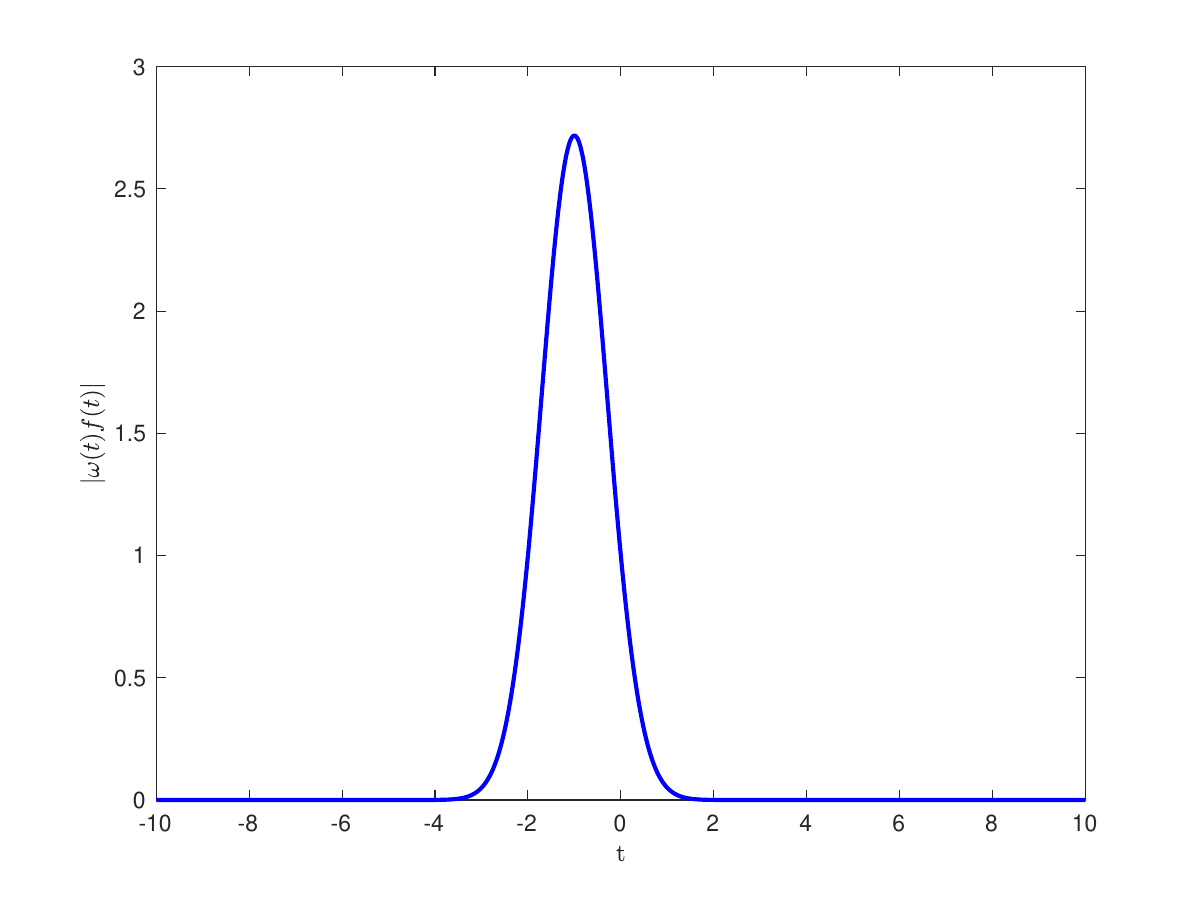}}
	\hfill
	\subfloat[Energy density of $\hat{f}(\xi)$]{
		\includegraphics[width=2.7 in]{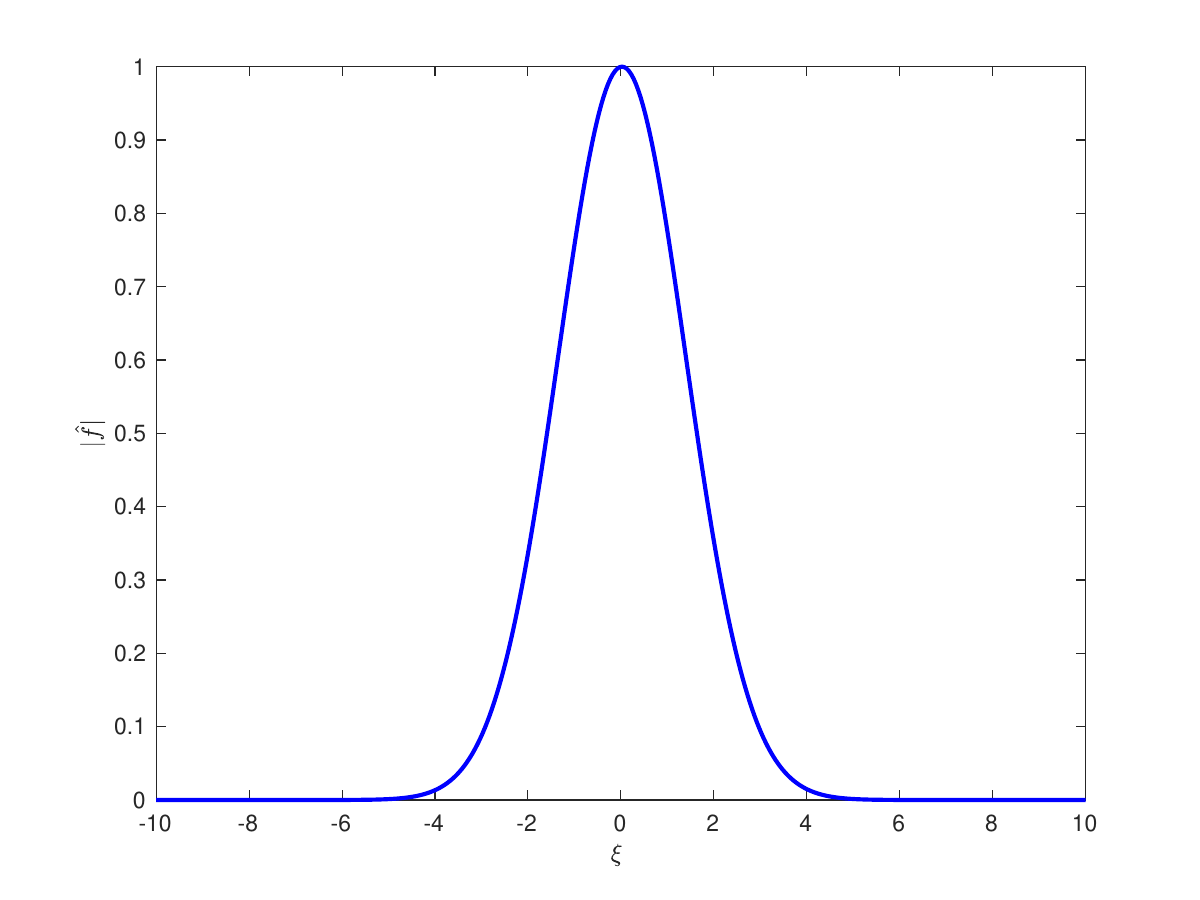}}
	\subfloat[Energy density of $O_J(\xi)$]{
		\includegraphics[width=2.7 in]{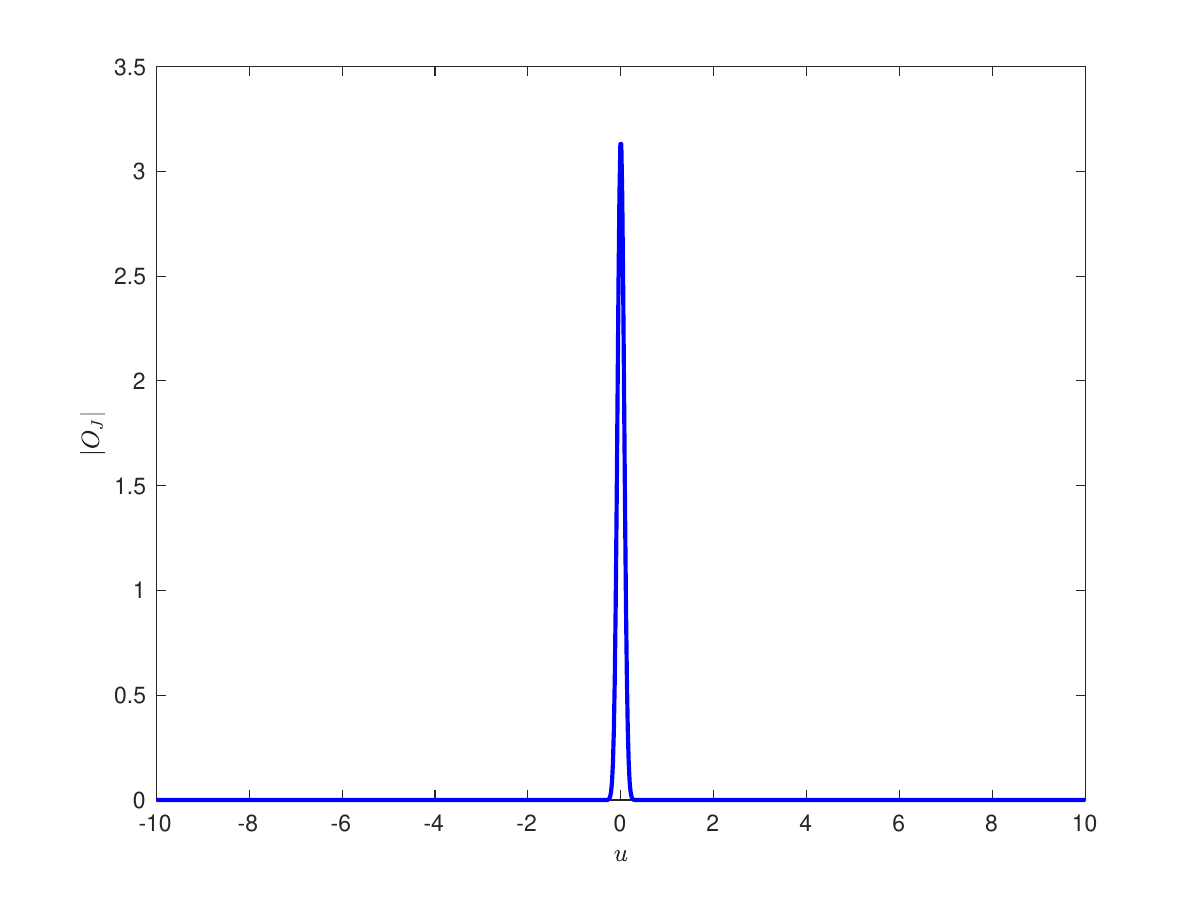}}
	\caption{The energy of a signal that can achieve the lower bound.}
	\label{Fig.3}
\end{figure}

As shown in Figure \ref{Fig.3}, when appropriate parameters are chosen for the OLCT, the signal energy exhibits greater concentration in the OLCT domain than in the conventional Fourier domain. And it is easy to calculate that
\begin{equation*}
\sqrt{(\mu_{2})_{\omega, \left| f(t) \right|^2}} \sqrt{(\mu_{2})_{ \left| O_f^J(\xi) \right|^2}} =1.904493221525881,
\end{equation*}
\begin{equation*}
\frac{|b|}{2} \left| E_{p,f}^* \right|  =1.904493221525881.
\end{equation*}
The equality $\sqrt{(\mu_{2})_{\omega, \left| f(t) \right|^2}} \sqrt{(\mu_{2})_{\left| O_f^J(\xi) \right|^2}}=\frac{|b|}{2} \left| E_{p,f}^* \right|$ validates the conclusion of Theorem \ref{thm3.4} for the case when $p=1$.

\begin{example}
	Consider signals in Example \ref{ex1} with $r=2$ and $r=10$, respectively. Then the signals are given by $f(t)=\mathrm{e}^{-t^2} \mathrm{e}^{-\mathrm{j}6t^2}$ and $f(t)=\mathrm{e}^{-5t^2} \mathrm{e}^{-\mathrm{j}6t^2}$. And the real weight functions are given by $\omega(t) = \mathrm{e}^{-2t}$ and $\omega(t) = \mathrm{e}^{-10t}$. Two settings of the parameters are considered: one with $a=0.6$, $b=0.05$ and the other with $a=6$, $b=0.5$, while keeping $c=0.5$, $d=0.4$, $\tau=0$, $\eta=1$ fixed.
\end{example}
To illustrate the influence of the weight function $\omega(t)$ and the OLCT parameters on the energy concentration of the signal, we present the energy distributions under different weight functions and various parameter settings.
\begin{figure}[t]
	\centering
	\subfloat[$\omega(t)=\mathrm{e}^{-2t}$]{  
		\includegraphics[width=2.7 in]{ex2w.pdf}}
	\subfloat[$\omega(t)=\mathrm{e}^{-10t}$]{
		\includegraphics[width=2.7 in]{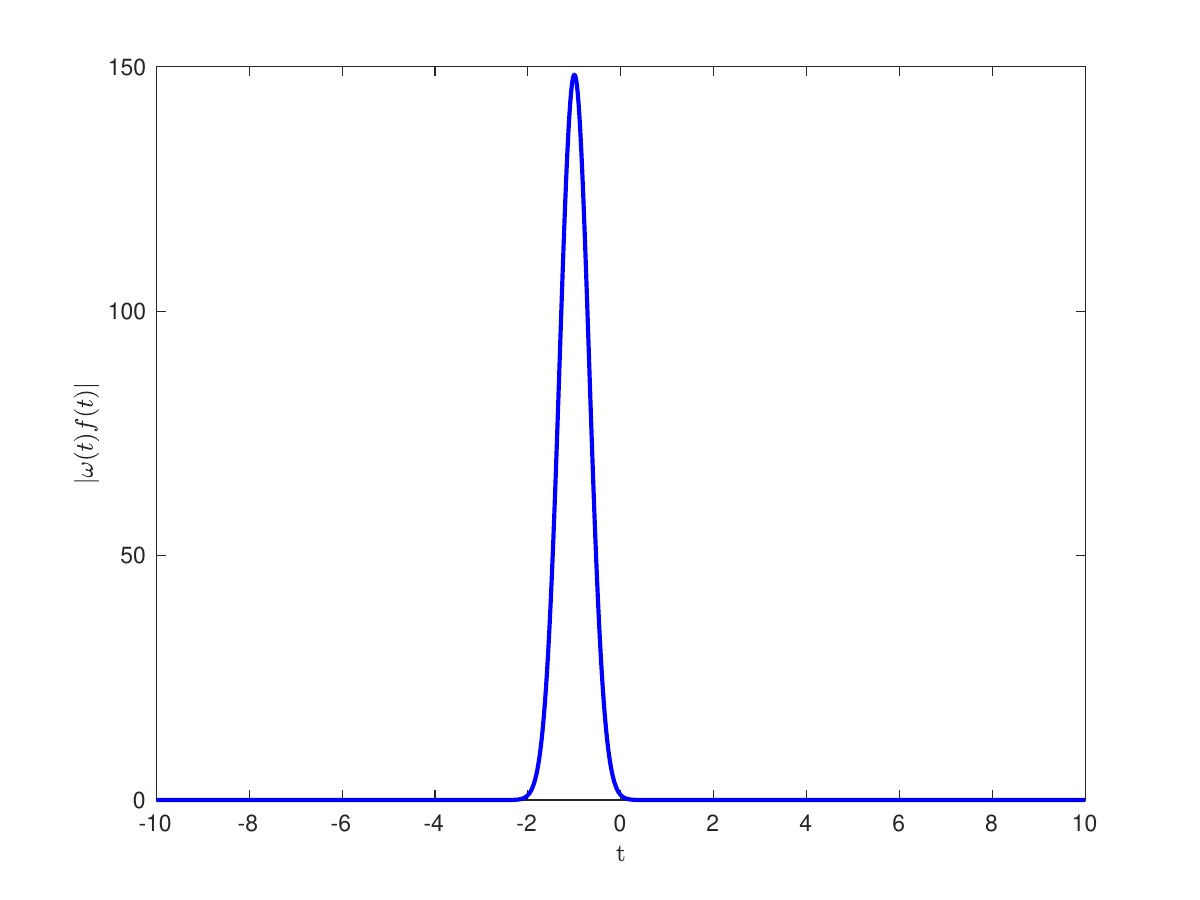}} 
	\vspace{1ex} 
	\subfloat[$a=0.6$, $b=0.05$]{
		\includegraphics[width=2.7 in]{ex2olct.pdf}}
	\subfloat[$a=6$, $b=0.5$]{
		\includegraphics[width=2.7 in]{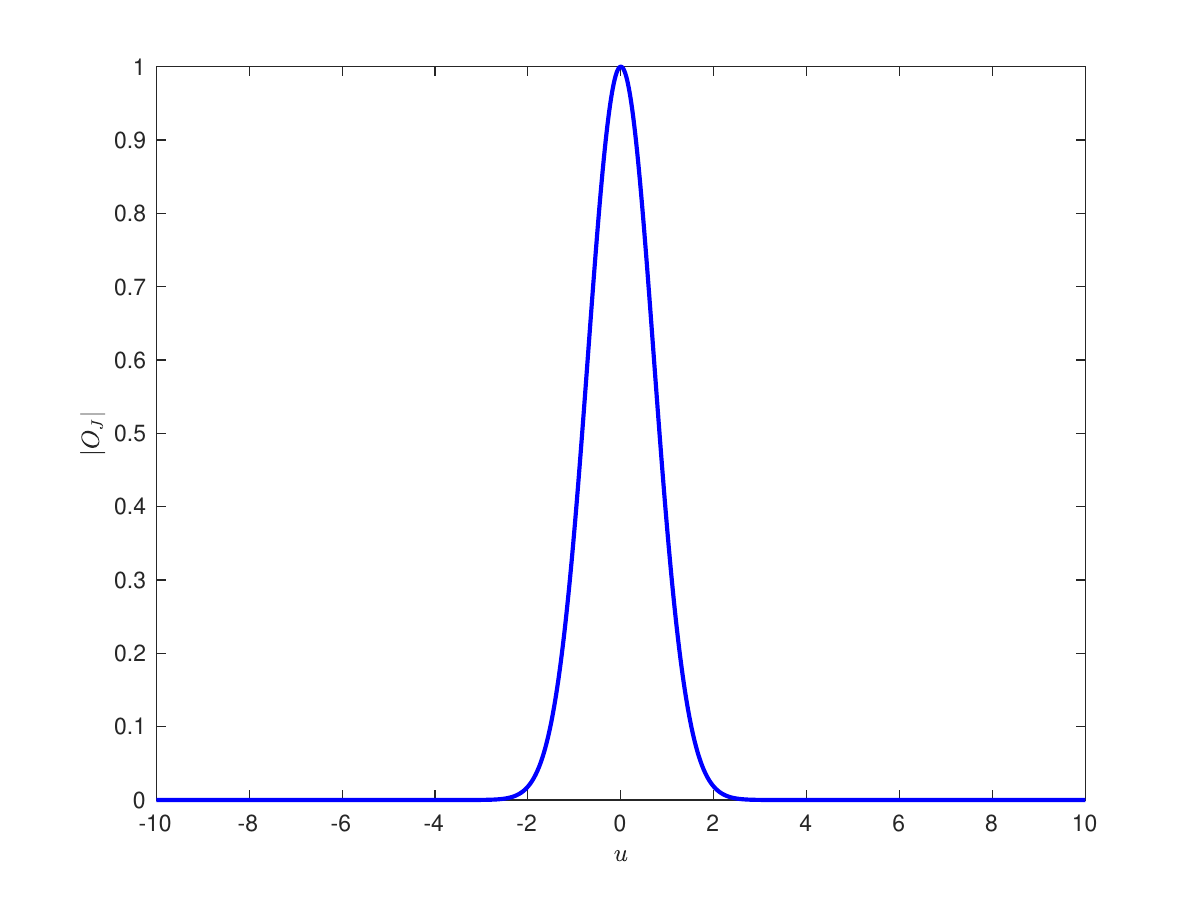}}
	\caption{Influence of the weight function $\omega$ and parameters $a$ and $b$ on the energy distribution of the given signal.}
	\label{fig4}
\end{figure}

Figure \ref{fig4} demonstrates that the concentration of the signal's energy is dependent on the choice of the weight function $\omega$ and the parameters $a$ and $b$ associated with the OLCT. In particular, a faster decay of the weight function $\omega$ leads to a more pronounced localization of the signal's energy. 

\section{Conclusion}
\label{sec5}
The uncertainty principle plays a pivotal role in signal processing and its applications. In this paper, the Plancherel-Parseval-Rayleigh identity in the offset linear canonical transform (OLCT) domain is derived. Based on this identity, the $2p$ order Heisenberg-Pauli-Weyl uncertainty principle of the OLCT is established. Subsequently, the Heisenberg-Weyl uncertainty principle in the OLCT domain is deduced. The sharpened Heisenberg-Weyl  uncertainty principle is then proposed and further validated through numerical simulations. The results of this work provide theoretical advancements and provide valuable insights for practical applications in time-frequency analysis within the OLCT domain. 

\section*{Acknowledgment}
This work was supported by grants from the National Natural Science Foundation of China [No. 62171041].

\section*{Appendix A.  Proof of Theorem \ref{thm3.2}}
To prove Theorem \ref{thm3.2}, we set
\begin{equation}
\begin{aligned}
\label{21}
M_p=&(\mu_{2p})_{\omega,\left|f(t) \right|^2}(\mu_{2p})_{\left| O_f^J(\xi) \right|^2} \\
=& \left( \int_{\mathbb{R}} \omega^2(t) (t-t_m)^{2p} \left|f(t) \right|^2 \mathrm{d}t \right) \cdot \left( \int_{\mathbb{R}} (\xi-\xi_m)^{2p} \left|O_f^J(\xi) \right|^2 \mathrm{d}\xi \right). 
\end{aligned}
\end{equation}
Since $g_{\beta}(t) = \mathrm{e}^{-\mathrm{j} \beta t} \mathrm{e}^{\mathrm{j} \frac{a}{2b} t^2}  f(t)$ and $\left| g_{\beta}(t) \right|^2 = \left|f(t) \right|^2$, then from \eqref{14} and \eqref{21}, we obtain
\begin{align}
\label{22}
M_p=&b^{2p} \left( \int_{\mathbb{R}} \omega^2(t) (t-t_m)^{2p} \left| g_{\beta}(t) \right|^2 \mathrm{d}t \right) \left( \int_{\mathbb{R}} \left| g_{\beta}^{(p)} (t) \right| ^2 \mathrm{d}t \right). 
\end{align}
Based on \eqref{22}, \eqref{91} and applying the Cauchy-Schwarz inequality, we get
\begin{equation}
\label{23}
M_p \ge b^{2p} \left( \int_{\mathbb{R}} \left| \omega_p(t) g_{\beta}(t) g_{\beta}^{(p)} (t) \right| \mathrm{d}t \right)^2.
\end{equation}
Applying the complex inequality to \eqref{23}, we have
\begin{equation}
\label{25}
M_p \ge \frac{b^{2p}}{4} \left( \int_{\mathbb{R}} \omega_p(t) \left( g_{\beta}(t) \overline{g_{\beta}^{(p)} (t)} + \overline{g_{\beta}(t)} g_{\beta}^{(p)} (t) \right) \mathrm{d}t \right)^2.
\end{equation}
Using the differential identity \eqref{6} on \eqref{25}, we obtain
\begin{equation}
\label{26}
M_p \ge \frac{b^{2p}}{4} \left( \int_{\mathbb{R}} \omega_p(t) \left( \sum_{q=0}^{\left[ \frac{p}{2} \right]} D_q \frac{\mathrm{d}^{p-2q}}{\mathrm{d}t^{p-2q}} \left| g_{\beta}^{(q)} (t) \right|^2 \right) \mathrm{d}t \right)^2.
\end{equation}
Applying the Lagrange type differential identity \eqref{9} to $\left| g_{\beta}^{(q)} (t) \right|^2$, we get
\begin{equation}
\begin{aligned}
\label{27}
\left| g_{\beta}^{(q)} (t) \right|^2= \sum_{n=0}^q B_{qn} \left|g^{(n)}(t) \right|^2 +2 \sum_{0 \leq i< z \leq q} C_{qiz} \mathrm{Re}((-1)^{q-\frac{i+z}{2}} g^{(i)}(t) g^{\overline{(z)}}(t)).
\end{aligned}
\end{equation}
From the integral identity \eqref{12} and \eqref{27}, we find
\begin{equation}
\label{28}
\int_{\mathbb{R}} \omega_p(t) \frac{\mathrm{d}^{p-2q}}{\mathrm{d}t^{p-2q}} \left|g^{(n)}(t) \right|^2 \mathrm{d}t = (-1)^{p-2q} \int_{\mathbb{R}}\omega_p^{p-2q}(t) \left|g^{(n)}(t) \right|^2 \mathrm{d}t = I_{qn},
\end{equation}
\begin{equation}
\begin{aligned}
\label{29}
&\int_{\mathbb{R}} \omega_p(t) \frac{\mathrm{d}^{p-2q}}{\mathrm{d}t^{p-2q}} \mathrm{Re}((-1)^{q-\frac{i+z}{2}} g^{(i)}(t) g^{\overline{(z)}}(t)) \mathrm{d}t \\
= &(-1)^{p-2q} \int_{\mathbb{R}}\omega_p^{p-2q}(t) \mathrm{Re}((-1)^{q-\frac{i+z}{2}} g^{(i)}(t) g^{\overline{(z)}}(t)) \mathrm{d}t \\
= &I_{qiz}.
\end{aligned}
\end{equation}
Taking \eqref{27}, \eqref{28} and \eqref{29} into \eqref{26}, we obtain
\begin{equation}
\begin{aligned}
\label{80}
M_p &\ge \frac{b^{2p}}{4} \left( \sum_{q=0}^{\left[ \frac{p}{2} \right]} D_q \left( \sum_{l=0}^q B_{qn} I_{qn} +2 \sum_{0 \leq i< z \leq q} C_{qiz} I_{qiz} \right) \right)^2 \\
&=\frac{b^{2p}}{4} \left( E_{p,f} \right)^2.
\end{aligned}
\end{equation}
This completes the proof of Theorem \ref{thm3.2}.


\end{document}